\documentclass[12pt, twoside, leqno]{article}

\usepackage{amsmath,amsthm}
\usepackage{amssymb}

\usepackage{enumerate}

\usepackage{graphicx}

\newtheorem{thm}{Theorem}[section]
\newtheorem{cor}[thm]{Corollary}

\theoremstyle{definition}
\newtheorem{defn}[thm]{Definition}
\newtheorem{rem}[thm]{Remark}
\newtheorem{exa}[thm]{Example}

\numberwithin{equation}{section}

\usepackage{pgf,tikz}
\usetikzlibrary{arrows}

\begin{document}

\baselineskip=12.1pt


\title{On the automorphism groups of $us$-Cayley  graphs }

\author{S.Morteza Mirafzal \\ Department of Mathematics \\
  Lorestan University, Khorramabad,  Iran\\   
E-mail: smortezamirafzal@yahoo.com,   mirafzal.m@lu.ac.ir}

\date{}

\maketitle


\renewcommand{\thefootnote}{}



\footnote{\emph{Date}:  }

\renewcommand{\thefootnote}{\arabic{footnote}}
\setcounter{footnote}{0}
\date{}
\begin{abstract}
Let $G$ be a  finite abelian group written additively with identity $0$, and
  $\Omega$ be an inverse closed generating  subset of   $G$ such that $0\notin \Omega$.  We say that $ \Omega $ has the   property  \lq\lq{}$us$\rq\rq{} (unique summation),  whenever   for every $0 \neq g\in G$ if  there are $s_1,s_2,s_3,
s_4 \in \Omega  $ such that  $s_1+s_2=g=s_3+s_4 $, then we have $\{s_1,s_2 \} = \{s_3,s_4 \}$. We say that a Cayley graph $\Gamma=Cay(G;\Omega)$ is a $us$-$Cayley\  graph$, whenever $G$ is an abelian group and the generating subset $\Omega$ has the property \lq\lq{}$us$\rq\rq{}.
In this paper, we show that if $\Gamma=Cay(G;\Omega)$ is a $us$-$Cayley\  graph$, then $Aut(\Gamma)=L(G)\rtimes A$, where $L(G)$ is the left regular representation of $G$  and $A$ is the group of all  automorphism groups $\theta$ of the group $G$ such that  $\theta(\Omega)=\Omega$. Then, as some applications,  we explicitly determine the automorphism groups of some classes of graphs  including  M\"{o}bius ladders and $k$-ary $n$-cubes.

\end{abstract}

\maketitle 
\section{ Introduction}
\noindent
  Let $G$ be a finite group with identity $1$  and
suppose that $\Omega$ is an inverse closed subset of   $G$ such that $1 \notin \Omega$. When the group $G$ is an abelian group, then we say that the Cayley graph $\Gamma=Cay(G;\Omega)$ is an abelian Cayley graph.
\begin{defn} Let $G$ be  a finite abelian group written additively   and
 $\Omega$ be an inverse closed subset of   $G$ such that $0 \notin \Omega$.  We say that $ \Omega $ has the  property   \lq\lq{}$us$\rq\rq{} (unique summation),   whenever:\\
  for every $0 \neq g \in G$ if there are $s_1,s_2,s_3,
s_4$  $ \in \Omega, $ such that  $s_1+s_2=g=s_3+s_4$, then we have $\{s_1,s_2 \} = \{s_3,s_4 \}$. We say that a Cayley graph $\Gamma=Cay(G;\Omega)$ is a $us$-$Cayley\  graph$, whenever $G$ is an abelian group and $\Omega$ has the property \lq\lq{}$us$\rq\rq{}.
\end{defn}
  There are   important   families of $us$-Cayley  graphs in graph theory and  its applications.

\begin{exa}Let $4\neq n \geq 3$ be an  integer. Consider the cycle $C_n$. We know that $C_n = Cay (\mathbb{Z}_n;\Omega )$, where $\mathbb{Z}_n$  is the cyclic group of order $n$ and $\Omega=\{1,-1  \}$. According to Definition 1.1. it is easy to see that $C_n$ is a $us$-Cayley  graph. Note that $C_4$ is not a $us$-Cayley  graph  because in $\mathbb{Z}_4$ we have $1+1=2=-1-1$,  but $\{1,1\} \neq \{-1,-1\}$.

 \begin{exa} Let $n\geq 1$ be an  integer.  The hypercube  $Q_n$ is the graph whose vertex set is $ \{0,1  \}^n $, where two $n$-tuples  are adjacent if  they differ in precisely one coordinates. It is   easy   to see that $Q_n   \cong Cay(\mathbb{Z}_{2}^n; \Omega )$, where   $\Omega=\{ e_i \  | \  1\leq i \leq n \}, $    $e_i = (0, ..., 0, 1, 0, ..., 0)$  with 1 at the $i$th position. It is  easy   to show that the hypercube $Q_n$ is a $us$-Cayley  graph.
\end{exa}

\end{exa} 

\begin{exa} Let $c,d,m$ be integers such that $1 \leq c <d$, $d \geq 5$ and $m\geq 1$. Let $n=cd^m$ and $\Omega=\{ \pm 1,...,\pm d^{m-1} \}$. Let $\Gamma=Cay (\mathbb{Z}_n;\Omega )$. It is easy to check that $\Gamma$ is a $us$-Cayley graph.

\end{exa}

 \begin{exa} Let $n\geq 4$. The  M\"{o}bius   ladder   on $n$ vertices, which we  denote it by $M_n$, is constructed by connecting vertices $u$ and
$v$ in the cycle $C_n$ if $d(u; v) = diam(C_n)$. Vertices that satisfy this condition are called
antipodal vertices. The
two cases ($n$ even or odd) change the structure of these graphs quite radically. If $n$ is an even
integer,  then the M\"{o}bius ladder
$M_n$ is called an even  M\"{o}bius ladder. Similarly,
If $n$ is an odd integer,  then the M\"{o}bius
 ladder    $M_n$ is   an odd M\"{o}bius ladder.
The even
M\"{o}bius ladder $M_{2k}$ is  a
 3-regular graph with diameter $k$,  and the odd  M\"{o}bius ladder $M_{2k+1}$ is a
 4-regular graph with diameter $k$. It is easy to see that $M_{2k} \cong Cay(\mathbb{Z}_{2k};S)$, where $S=\{1,2k-1,k  \}$, and $M_{2k+1} \cong Cay(\mathbb{Z}_{2k+1};T)$, where $T=\{1,2k,k,k+1\}.$
  When $n=6=2\times 3$, then in $\mathbb{Z}_6$ we have $3+5=1+1$.
Hence,     From Definition 1.1  it follows that if  $n=6$, then the even M\"{o}bius ladder
  $M_n$ is not a $us$-Cayley graph.
  It can be easily shown that if $n=2k, \ k>3$, then  the M\"{o}bius ladder
  $M_n$ is a $us$-Cayley graph.
   On the other hand,
 When $n=5=2\times 2+1$, then in $\mathbb{Z}_5$ we have $4+4=1+2$. Also, when $n=7=2\times 3 + 1$, then in $\mathbb{Z}_7$  we  have $6+3=1+1$. Therefore,     from the definition of   $us$-Cayley graphs  it follows that when $n\in \{ 5,7 \}$, then the odd M\"{o}bius ladder
  $M_n$ is not a $us$-Cayley graph. On the other hand,
  it can easily be  shown that if $n=2k+1, \ k>3$,  then  the M\"{o}bius ladder
  $M_n$ is a $us$-Cayley graph.

\end{exa}
 We have not found any document containing a proof for determining the automorphism group of the  odd M\"{o}bius ladder  $M_{2k+1}$.

\begin{exa}
Let $k$ and $n$  be integers with $k>1$ and $n>0$.
The $k$-ary   $n$-cube  $ Q_n ^k$
  is defined as a graph with vertex set $ \mathbb{Z}_k^n $, in which two vertices $v =(x_1,x_2,...,x_n)$ and $w =(y_1,y_2,...,y_n)$
are adjacent if and only if there exists an integer $j$,  
$1 \leq j \leq n$, such that $x_j = y_j \pm 1$ (mod k) and $x_i = y_i$ for
all $i \in  \{1, 2,..., n\}-\{j\}$. It is easy to see that  that
$ Q_n ^k = Cay(\mathbb{Z}_{k}^n; \Omega )$, where $\mathbb{Z}_{k}$ is
the cyclic group of order $k$, and $\Omega=\{ \pm e_i \  | \  1\leq i \leq n \}, $ where  $e_i = (0, ..., 0, 1, 0, ..., 0)$,  with 1    at the $i$-th position. It is an easy task to show that the $k$-ary   $n$-cube  $ Q_n ^k$  is a $us$-Cayley  graph whenever $k \neq 4$.

\end{exa}

To the best of our knowledge, the automorphism groups of the graphs in Example 1.6. have not been determined.\

In this paper, we determine the automorphism group of $us$-Cayley  graphs. In fact we show that
 if $\Gamma=Cay(G;\Omega)$ is a $us$-$Cayley\  graph$, then $Aut(\Gamma)=L(G)\rtimes A$, where $L(G)=\{l_v | v \in  G, \   l_v(x)=v+x$ for every $ x \in G  \}$  and $A$ is the group of all group automorphisms $\theta $ of the group $G$ such that  $\theta(\Omega)=\Omega$. This result generalizes known facts about automorphism groups of some families of graphs to some other classes of graphs.\

Let $\Gamma=(V,E)$ be a graph. There is an intimate  relationship between some topological properties of the graph $\Gamma$ and its automorphism group $Aut(\Gamma)$. For instance if $Aut(\Gamma)$ acts transitively on the sets $V$ and $E$, then the vertex-connectivity of $\Gamma$, which is an important parameter in applied graph theory, is best possible [16]. 
\section{Preliminaries}
In this paper, a graph $\Gamma=(V,E)$ is
considered as a finite undirected simple graph where $V=V(\Gamma)$ is the vertex-set
and $E=E(\Gamma)$ is the edge-set. For all the terminology and notation
not defined here, we follow $[1,3,5]$.\

The graph $\Gamma$ is called $vertex$-$transitive$  if  $Aut(\Gamma)$
acts transitively on $V(\Gamma)$.  For $v\in V(\Gamma)$ and $G=Aut(\Gamma)$ the stabilizer subgroup
$G_v$ is the subgroup of $G$ consisting of all automorphisms that
fix $v$. We say that $\Gamma$ is $symmetric$ (or $arc$-$transitive$) if for all vertices $u, v, x, y$ of $\Gamma$ such that $u$ and $v$ are adjacent, also, $x$ and $y$ are adjacent, there is an automorphism $\pi$ in $Aut(\Gamma)$ such that $\pi(u)=x$ and $\pi(v)=y$. It is clear that a symmetric graph is vertex   and edge transitive. It is a known fact that a vertex-transitive graph $\Gamma$ is symmetric if and only if for a vertex $v \in V(\Gamma)$ the subgroup $G_v$ acts transitively on the set $N(v)$ [1.  Chapter 15,  4.  Chapter 4].

Let $G$ be any abstract finite group with identity $1$, and
suppose that $\Omega$ is an inverse closed subset of   $G$, namely, a subset of $G$  with the following properties:

(i) $x\in \Omega \Longrightarrow x^{-1} \in \Omega$,   $ \ (ii)
 \ 1\notin \Omega $.

The $Cayley\  graph$  $\Gamma=\Gamma (G; \Omega )$ is the (simple)
graph whose vertex-set and edge-set defined as follows: $V(\Gamma) = G $,  $  E(\Gamma)=\{\{g,h\}\mid g^{-1}h\in \Omega \}$.\

A group $G$ is called a semidirect product of $ N $ by $Q$,
denoted by $ G = N \rtimes Q $,
 if $G$ contains subgroups $ N $ and $ Q $ such that, (i)\ $
N \unlhd G $ ($N$ is a normal subgroup of $G$ ); (ii) $ NQ = G $; and
(iii) $N \cap Q =\{1\} $. \ \

Let $H$ be a group. The $left\  regular\  representation $   of $H$ denoted by $L(H)$ 
is the permutation group $\{l_h |\  h\in H\}$ in $Sym(H)$, where $l_h$ is defined to be the
map $l_h(x): H \rightarrow H$,  $l_h(x)=hx$ for every $ x\in H$. Recall that Cayley's theorem asserts that every
group $H$ is isomorphic to the permutation group $L(H)$. Let $\Gamma=Cay(H;\Omega)$ be a Cayley graph. If $x,y \in H$ and $x \leftrightarrow y$, that is,  $x$ adjacent to $y$ in $\Gamma$,  then $x^{-1}y \in \Omega$. Thus for every $h \in H$ we have  ${(hx)}^{-1}(hy) \in \Omega$. Hence, $l_h$ is an automorphism of the graph $\Gamma$.
Therefore, the left regular representation $L(H)$ is a subgroup of the group $G$=$Aut(\Gamma)$.
Given a group $H$ and a subset $S \subseteq H$, let $Aut(H, S)$ be the subgroup  of
automorphisms of the group $H$ that fix the set  $S$ setwise. In other words,
$Aut(H, S)= \{g \in Aut(H) |\  g(S)  = S\}$. It is easy to see that
if $\Gamma=Cay(H;S),$ then $Aut(H, S)$ is a subgroup of the  stabilizer group $G_1$, where $G= Aut(\Gamma)$ and $1$ is the identity element of the group $H$. Let $K$ be a group and $T$ is a nonempty subset of $K$. The normalizer $N_K(T)$ of $T$ in $K$ is the subgroup of the elements $x$ in $K$ such that $x^{-1}tx \in T,  $ for evert $t \in T.$ It is easy to see that in the group $G=Aut(\Gamma)$, the group $Aut(H, S)$ is a subgroup of normalizer $L(H).$ In other words, if $h\in   H$ and  $f \in Aut(H, S)$ then $f^{-1}l_hf  \in L(H). $  Note that for each $x \in H$ we have $(f^{-1}l_hf)(x)=f^{-1}(hf(x))=f^{-1}(h)x$, and hence $f^{-1}l_hf=l_{f^{-1}(h)}.$ Also note that in the group $G$ we have $L(H) \cap Aut(H, S)=\{ 1\}.$  Therefore    $\langle L(H),Aut(H, S) \rangle$, the generated subgroup of $L(H)$ and $Aut(H, S)$ in $G$,  is the subgroup $L(H) \rtimes Aut(H, S)$. Hence,  we have $L(H) \rtimes Aut(H, S) \leq G=Aut(\Gamma). $ \

\section{Main results}

There are various important families of graphs $\Gamma$,  in which we know that for a particular group $M$ we have
$M \leq Aut(\Gamma)$, but  showing  that in fact    $M = Aut(\Gamma)$  is not an easy task. 
The interested reader can find various examples about this matter in graph theory (see [7,8,9,10,11,12,13,14]). \\
In the sequel, we wish to show that if $\Gamma=Cay(H;S)$ is a $us$-Cayley graph, then $L(H) \rtimes Aut(H, S) =  Aut(\Gamma).$

\begin{thm} Let $H$ be an abelian group written additively and $S$ be a generating inverse closed subset of $H$ with $ 0 \notin S$.  Let $\Gamma=Cay(H;S) $ be a $us$-Cayley graph,  $G=Aut(\Gamma)$ and $G_0$ be the  subgroup of
$G$ consisting of automorphisms of $\Gamma$ which fix  the vertex $0$. Then $G_0=Aut(H,S)$, where $Aut(H,S)$ is the subgroup of automorphisms of the group $H$ which stabilize the set $S$.   In other words, every automorphism $f$ of the graph $\Gamma$ which fixes the vertex $v=0$ is a group automorphism of the abelian group $H$ such that $f(S)=S$.
\end{thm}

\begin{proof}
It is clear that $Aut(H)_{{S}} \leq G_0$, hence it is sufficient to show that  $G_0 \leq Aut(H)_{{S}}$.
Let $f \in G_0$.  We show that $f$ is an automorphism of the group $H$ such that $f(S)=S$.  We prove the theorem in some steps.\\
(1)  In the first step, let $v,w \in S, \ v \neq w  $ and $v + w \neq 0$. We wish to show that $f(v+w)=f(v)+f(w).$\\
    Since $v + w - v = w \in S$ and $v + w - w = v \in S, $   then $v + w$ is adjacent to vertices $v$ and  $w$. In other words,  $\{v + w, v\} \in E( \Gamma)$ and $\{v + w, w\} \in E( \Gamma)$. Now since $f$ is an automorphism of the graph $\Gamma$, then
     $\{f(v + w), f(v)\} \in E( \Gamma)$ and   $\{f(v + w), f(w)\} \in E( \Gamma)$.  Therefore  $f(v + w) -f(v) \in S $ and  $f(v + w)-f(w) \in S$.
     Thus  $0\neq f(v + w) =f(v) + s_1$
     and $ f(v + w) = f(w) + s_2$,  where  $s_1, s_2 \in S$. Hence we have  $0 \neq f(v)+s_1=f(w)+s_2$. Note that if $u \in S$, then $u-0=u \in S$, hence $\{0,u \}\in E(\Gamma)$,
      so that $\{f(0),f(u)\} = \{0, f(u)\}\in E(\Gamma)$.   Now since  $N(0) =S$ ($N(x)$ is the set of vertices of $\Gamma$
      which are adjacent to the vertex $x$),  then  $f(u) \in S$, and  thus  $f(S) = S$.
       We now have $0 \neq f(v) + s_1 = f(w) + s_2$ with $f(v), s_1,f(w), s_2 \in S$. Now, since $\Gamma=Cay(H;S) $ is a $us$-Cayley graph, then we have  $\{f(v) , s_1  \}=\{ f(w) , s_2 \}$. Now, nothing that $v \neq w  $ and $f$ is a permutation of the set $S$,  we conclude that $s_2=f(v),  s_1=f(w).$  Therefore,  we have  $f(v + w) = f(v) + s_1 = f(v) + f(w).  $ \

(2)  We now  show that if $u \in S$,   then  $f(u+u)=f(2u) = f(u)+f(u)=2f(u)$. In the first step, we assume  $2u
\neq 0$. Then,    $f(2u) \neq 0$.
  Since $2u - u = u \in S$,   then   $\{2u,u \} \in E(\Gamma)$ so that
$\{f(2u),f(u)\} \in E(\Gamma)$.  Thus   $f(2u) = f(u) + y$, where
$y \in S = f(S)$.   Therefore,  there
 is an $x \in S$ such that $y = f(x)$ and hence  we have $f(2u) = f(u) + f(x)$. We assert that $f(x) = f(u)$. On the contrary, assume that
 $f(x) \neq f(u)$.  Then,  since   $f^{-1}\in G_0$, and  by what is proved in (1), we have
     $2u = f^{-1}(f(2u)) = f^{-1}(f(u) + f(x)) = f^{-1}(f(u)) + f^{-1}(f(x)) = u + x$.
    Thus,  $x = u$, and hence we have  $f(x) = f(u)$,  which is  a contradiction. Therefore,  if $u \in S$ and $2u \neq 0$,   then  $f(2u) = 2f(u)$. \newline Now  let $ 2u = 0 $. We claim that $2f(u)=0=f(2u).$ Note that if $f(u) =
u$,  then  $2f(u) = 2u = 0 = f (0) = f(2u)$.
On the contrary,  assume that $2f(u) \neq 0 $.  Then,  by what is proved hitherto, we have  $f^{-1}(2f(u)) = 2f^{-1}(f(u)) = 2u=0=f^{-1}(0)$.
 Noting that $f^{-1}$ is a permutation of the set $H$, we deduce that $2f(u)=0$, which is a contradiction. We now conclude that $2f(u)=0.$
    \

(3) We now  show that if $x \in S$,   then   $f(-x) = - f(x)$. In the first step, we assume that  $2x
\neq 0$.
 Then $x \neq -x$, hence $f(x) \neq f(-x)$.     Thus,  if $ t = f(x) + f(-x)\neq 0$,
 then  by what is proved in (1),
 we have $f^{-1}(t) = f^{-1}(f(x) + f(-x)) =f^{-1}(f(x)) + f^{-1}(f( -x)) = x -x = 0 = f^{-1}(0)  $.
 Thus $t = 0$ which is  a contradiction.   Therefore,  we must
have $f(x) + f(-x) = 0$, and hence $f(-x) =- f(x) $. \\
 We now assume that $ 2x = 0$.
Then  $x = -x$. Now by what is proved in (2), we have  $f(2x) = f(0) =0=  2f(x)$.   Thus,    $f(x) = - f(x) =-f(- x)$ which implies that $f(-x) = - f(x)$.\\
Moreover, note that
 if $v,w \in S$ and $v+ w = 0$,   then   $w = -v$.  Hence,   we have
  $f(v + w) = f(0) = 0 = f(v) - f(v) = f(v) + f(-v)=f(v) + f(w)$. \

(4) We now show that for any $k\in \mathbb{N},\   k>1$, if $v=x_1+x_2+x_3+...+x_k$, where $x_i \in S, \  1\leq i \leq k,  $ then we have $f(v)=f(x_1)+f(x_2)+f(x_3)+\cdots+f(x_k)$. We prove by induction on $k$. If $k=2, $  then by what is proved hitherto, the assertion is true. Assume $k>2$. In the first step, assume that all the ${x_j}^,s$ are not identical, namely, there are some $x_r,x_t$ such that $x_r \neq x_t$. Then, we have $v=u+x_r + x_t,$ where $u = v-(x_r + x_t)$ and hence $u$ is a sum of $k-2$ elements of $S$. Since $ v-(u+x_r) = x_t$, then $\{v, (u+x_r) \} \in E(\Gamma),$ and  hence $\{f(v), f(u+x_r) \}\in E(\Gamma)$. Therefore,  $f(v)=f(u+x_r)+ f(s_1),  $  for some $s_1 \in S.$  Similarly, Since $ v-(u+x_t) = x_r$, then we have $f(v)=f(u+x_t)+ f(s_2),  $  for some $s_2 \in S$. By induction hypothesize we have, $ f(u+x_r)=f(u)+f(x_r)$  and $ f(u+x_t)=f(u)+f(x_t)$. Therefore we have,
$f(v)=f(u+x_r)+ f(s_1) = f(u)+f(x_r) + f(s_1) = f(u)+f(x_t) + f(s_2). $ We now deduce that $f(x_r) + f(s_1) = f(x_t) + f(s_2). $ Equivalently,  we have $f(x_r +   s_1) = f(x_t  +  s_2). $ Noting that $f$ is a permutation of $H=V(\Gamma)$, it follows that $x_r +   s_1 = x_t  +  s_2.$
Now by assumption of our theorem we have  $\{x_r, s_1  \} = \{x_t, s_2  \}. $ Noting that $x_r \neq x_t$, we have $ x_r= s_2, x_t= s_1,  $ and we have;\\

\centerline{$f(v)=f(u+x_r + x_t) = f(u+x_r)+ f(s_1)=f(u)+f(x_r) + f(x_t)$}
 \ \newline   Now, since by induction hypothesize, $u$ is a sum of $k-2 $ element of $S$, we  conclude that; \

\

\centerline{$f(v)=f(x_1)+f(x_2)+f(x_3)+\cdots+f(x_k)$}

  \  \newline We now consider the case $v=x_1+x_2+x_3+...+x_k$, where $x_i \in S, \  1\leq i \leq k,  $ and all of the ${x_i}^,s$ are identical. In other words, we want to show that if   $v=kx_1, $ then $f(v)=kf(x_1).  $  Nothing that $kx_1-(k-1)x_1 \in S, $ we have,   $\{kx_1, (k-1)x_1\} \in E(\Gamma)$.  Hence,  $\{f(kx_1), f((k-1)x_1) \in E(\Gamma) \}.  $ Thus $f(kx_1)= f((k-1)x_1) +f(s), $ for some $s \in S. $ We show that $s= x_1. $ In fact if $s \neq x_1$, then by what is proved heretofore we must have;  \

\

\centerline{$f(kx_1)= f((k-1)x_1)  +f(s)= f( (k-1)x_1+s) $}
\  \newline Hence,  $kx_1=(k-1)x_1 +s,  $ and thus $x_1= s, $  which is a contradiction. Now since $s=x_1, $ then by induction hypothesize we have $f(kx_1)= f((k-1)x_1)  +f(s)=(k-1)f(x_1)+f(x_1)=kf(x_1).$
\\Now, Nothing that $S$ is a generating subset of the abelian group $H$,   the theorem has been proved.
\end{proof}
We  now have enough tools to prove the following important result.

\begin{thm}Let $H$ be an abelian group written additively and $S$ is a generating inverse closed subset of $H$ with $ 0 \notin S$.  Let $\Gamma=Cay(H;S) $ be a $us$-Cayley graph.
Then   $Aut(\Gamma)=L(H) \rtimes A$,
  where $L(H) $ is the left regular representation of  the group $H$ and $A=Aut(H,S)$  is the group of all   automorphisms of the group $H$ which stabilizes the set $S$.

\end{thm}

\begin{proof}
Let $\Gamma = Cay(H;S)$ and let $G=Aut(\Gamma)$ be the automorphism group of the graph $\Gamma$. For the vertex $v=0$ in the graph $\Gamma$,  let  $G_0$ be the stabilizer subgroup of this vertex. The graph $ \Gamma $ is a  vertex-transitive graph
because it is a Cayley graph.  Thus,
 by the orbit-stabilizer theorem we have $|V(\Gamma)|=\frac{|G|}{|G_0|}.$ Therefore, $|G|=|V(\Gamma)||G_0| $.
 It is clear that if $L=L(H)=\{l_h | h \in  H \}$, $l_h(x)=h+x$ for every $x \in H$, then $|L|=|H|=|V(\Gamma)|$. Indeed, it is easy to see that $L(H) \cong H$.
   From  Theorem 3.1. we have  $G_0=Aut(H)_{{S}} = A$, and hence $|G|=|L||A|.$  It is easy to see that   $L \cap A = \{e\}$, where $e$ is the identity automorphism of the graph $\Gamma$.
We assert that $AL=\{al | \ a\in A, \ l \in L  \}$ is a subgroup of the group $G$.  It is sufficient to show that $AL=LA$. If $a\in A$ and $l \in L$, then there is some $v \in H$ such that $l= l_v$. Thus, for every $x \in H$ we have  $(al)(x)=a(l_v(x))=a(v+x)=a(v)+a(x)=l_{a(v)}(a(x))= (l_{a(v)}a)(x)=(l_1a)(x)$, where $l_{a(v)}=l_1 \in L$. Hence, $al=l_1a \in LA $, thus $AL \subseteq LA$. Since $|AL|=|LA|$, therefore $AL=LA$.  We now deduce that $AL$ is a subgroup of the group $G$.  We know that  $|AL|= \frac{|A||L|}{|A\cap L|}=|A||L|=|G|$.  Therefore $G=AL.$ It is easy to see that $L$ is a normal subgroup of the group $G$, hence we have $Aut(\Gamma)=G=L\rtimes A$.
\end{proof}

\

{\bf Some applications }\\

We now show how Theorem 3.2. can help us in determining the automorphism groups of the graphs
which are appeared in the introduction section  of this paper.

\begin{thm} Let $c,d,m$ be integers such that $1 \leq c <d$, $d \geq 5$ and $m\geq 1$. Let 
$n=cd^m$ and $\Omega_1=\{ \pm 1,...,\pm 
d^{m-1} \}$. Let $\Omega=\{1,-1\} \cup B$ where $B$ is an inver closed subset of $\Omega$.  Let $\Gamma=Cay (\mathbb{Z}_n;\Omega )$.     Then $ 
Aut(\Gamma) \cong \mathbb{D}_{2n} $, where $\mathbb{D}_{2n} $ is the dihedral group of order $2n$.
\end{thm}

\begin{proof}
It is easy to check  that $\Gamma$ is a $us$-Cayley  graph.
Hence, from Theorem 3.2. it follows that $\Gamma = L(\mathbb{Z}_n) \rtimes A$ where $A=\{f|f\in Aut(\mathbb{Z}_n), f(\Omega)=\Omega  \}$. If  $f\in A$, then $f(1)$ is 
generating element of the cyclic group $\mathbb{Z}_n$, because $1$ is a generating element of the cyclic group  $\mathbb{Z}_n$. Since $f(\Omega)=\Omega$  and 1 and -1 are 
unique elements of $\Omega$ which can  generate the cyclic group $\mathbb{Z}_n)$, it follows that $f(1)=1$ or $f(1)=-1$. Note that every 
automorphism of the group $\mathbb{Z}_n$ is determined, when $f(1)$ is defined.  Now, it follows that $|A|=2$, hence $A \cong \mathbb{Z}_2$. Noting that $ L(\mathbb{Z}_n) \cong \mathbb{Z}_n$, we conclude that $ Aut(\Omega) \cong \mathbb{Z}_n \rtimes \mathbb{Z}_2 \cong \mathbb{D}_{2n}$.

\end{proof}

By a similar argument, we can show that  if 
  $4\neq n\geq 3$ be an integer,  then $ Aut(C_n) \cong \mathbb{D}_{2n} $, where $\mathbb{D}_{2n} $ is the dihedral group of order $2n$.
 
In the sequel, we need the  following fact [2].

\begin{thm}
Let $\Gamma $ be a graph with $n$ connected components  $\Gamma_{1},\Gamma_{2},\cdots,\Gamma_{n}$, where $\Gamma_{i}$ is isomorphic to $\Gamma_1$ for all $i \in [n]= \{1,\cdots,n \}=I$. Then  we have  $Aut(\Gamma)=Aut(\Gamma_1)wr_{I}\mbox{Sym}([n])$.
\end{thm}
We now provide a proof which shows that $Aut(M_n) \cong \mathbb{D}_{2n}$, where $M_n$ is the   M\"{o}bius ladder of order $n>6$.
\begin{thm} Let $n>7$ be an integer. Then $Aut(M_n) \cong \mathbb{D}_{2n}$, where $M_n$ is the M\"{o}bius ladder of order $n$ and $\mathbb{D}_{2n}$ is the dihedral group of order $2n$.

\end{thm}

\begin{proof} We know that there are two distinct cases, that is, (i) $n$ is an even integer, and (ii) $n$ is an odd integer.\

(i) Let $n=2k$ be an even integer, where $k>3$. We know that $M_{2k} \cong Cay(\mathbb{Z}_{2k};S)$, where $S=\{1,-1=2k-1,k \}$. From Theorem 3.2. we have $Aut(M_{2k}) \cong L(\mathbb{Z}_{2k}) \rtimes A$, where $A=\{ f \in Aut(\mathbb{Z}_{2k}) | f(S)=S \}$.
If  $f\in A$, then $f(1)$ is generating element of the cyclic group $\mathbb{Z}_{2k}$, because $1$ is a generating element of the cyclic group  $\mathbb{Z}_{2k}$. Since $f(S)=S$, it follows that $f(1)=1$ or $f(1)=-1$, because $k$ is a divisor of $n$, hence $k$ can not be a generating element of the cyclic group  $\mathbb{Z}_{2k}$. We now conclude, by what is done in the proof of Proposition 3.3. that
$Aut(M_n)=Aut(M_{2k}) \cong \mathbb{Z}_{2k} \rtimes \mathbb{Z}_{2} \cong \mathbb{D}_{4k}=\mathbb{D}_{2n}$. \

  (ii)  Let $n=2k+1$ be an odd integer, where $k\geq4$. We know that $M_{2k+1} \cong Cay(\mathbb{Z}_{2k+1};S)$, where $S=\{1,-1=2k,k,k+1 \}$. From Theorem 3.2. we have $Aut(M_{2k+1}) \cong L(\mathbb{Z}_{2k+1}) \rtimes A$, where $A=\{ f \in Aut(\mathbb{Z}_{2k+1}) | f(S)=S \}$.
If  $f\in A$, then $f(1)$ is generating element of the cyclic group $\mathbb{Z}_{2k+1}$, because $1$ is a generating element of the cyclic group  $\mathbb{Z}_{2k+1}$. Since $f(S)=S$, it follows that $f(1)\in S$. Note that in $\mathbb{Z}_{2k+1}$ we have $-k=k+1.$\\
 We claim that $f(1) \neq k$. On the contrary, let  $f(1)=k$.  Then $f(-1)=-k=k+1$.
Thus $f(k)=1$, or $f(k)=-1$.  \\
 If $f(k)=1$, then we have
  $-1=f(-k)=f(k+1)=f(k)+f(1)=1+k=-k$, and so $1=k$, which is impossible. If $f(k)=-1$, then we have
  $1=-f(k)=f(-k)=f(k+1)=f(k)+f(1)=-1+k$, and hence $2=k$, which is also impossible.\\
  We also assert that, $f(1) \neq k+1$. On the contrary, let  $f(1)=k+1$.  Then $f(-1)=-(k+1)=k$.
Thus $f(k)=1$, or $f(k)=-1$. \\
 If $f(k)=1$, then we have
  $-1=f(-k)=f(k+1)=f(k)+f(1)=1+k+1$, and thus $k=-3$.   Therefore, $3=-k=k+1$, thus $k=2$,  which is impossible. \\
   If $f(k)=-1$, then we have
  $1=-f(k)=f(-k)=f(k+1)=f(k)+f(1)=-1+k$, and hence $2=k$, which is also impossible.\

  We now deduce that $f(1)=1$ or $f(1)=-1$. Therefore,  by what is done in the proof of Proposition 3.3. we conclude that
$Aut(M_n)=Aut(M_{2k+1}) \cong \mathbb{Z}_{2k+1} \rtimes \mathbb{Z}_{2} \cong \mathbb{D}_{4k+2}=\mathbb{D}_{2n}$.

\end{proof}

\begin{rem}It is easy to  see that if $n\in \{ 4,5 \}$ then $M_n \cong K_n$, the complete graph of order $n$. Therefore, in these cases we have $Aut(M_n) \cong Sym([n])$, where $[n]=\{1,2,...,n \}$.\\
It is easy to see that if $n=7$, then $\overline{M}_n \cong C_7$, where $\overline{M}_n$ is the complement of the graph $M_n$. We know that if $G$ is a graph then $Aut(\overline{G})=Aut(G)$. Therefore, from Proposition 3.3. it follows that if $n=7$, then
$Aut(M_n) \cong \mathbb{D}_{2n}$. In other words, we have shown that if $n>6$ then $Aut(M_n) \cong \mathbb{D}_{2n}$.\\
Also,  It is easy to see that if $n=6$, then $\overline{M}_n$ is the disjoint union of two 3-cycle. Therefore, since $Aut(C_3) \cong Sym([3])$, then from Theorem 3.4.   we conclude that $Aut(M_6) \cong Sym([3])wr_ISym([2])$, where $I=\{1,2\}$.
\end{rem}

We now proceed to determine the automorphism group of a class of graphs which is important in some aspects in interconnection networks. An interconnection network can be represented as an undirected graph
where a processor is represented as a  vertex  and a communication
channel between processors as an edge between corresponding
vertices. Measures of the desirable properties for interconnection
networks include degree, connectivity, scalability, diameter, fault
tolerance, and symmetry.  The main aim of this subsection is to study the
symmetries of some classes of Cayley graphs. Note that  in designing    interconnection networks Cayley graphs play a vital role.
As we mentioned in Example 1.6. for the graph $k$-ary   $n$-cube  $ Q_n ^k$ we have 
$ Q_n ^k \cong Cay(\mathbb{Z}_{k}^n; \Omega )$, where $\mathbb{Z}_{k}$ is
the cyclic group of order $k$, and $\Omega=\{ \pm e_i \  | \  1\leq i \leq n \}, $ where  $e_i = (0, ..., 0, 1, 0, ..., 0)$,  with 1    at the $i$-th position. We can easily check    that the $k$-ary   $n$-cube  $ Q_n ^k$  is a $us$-Cayley  graph when $k \neq 4$.
   Note that each vertex of $ Q_n ^k$
has degree $2n$ when $k \geq 3$,
and $n$ when $k = 2$. Obviously, $ Q_1 ^k$
is a cycle of length $k$, $ Q_n ^2$ is
an $n$-dimensional hypercube. Figure 1.  illustrates\
$  Q_1^6,   \  Q_2^5,$   and $  Q_3^3$. \

\begin{figure}[ht]
\centerline{\includegraphics[width=7.9 cm]{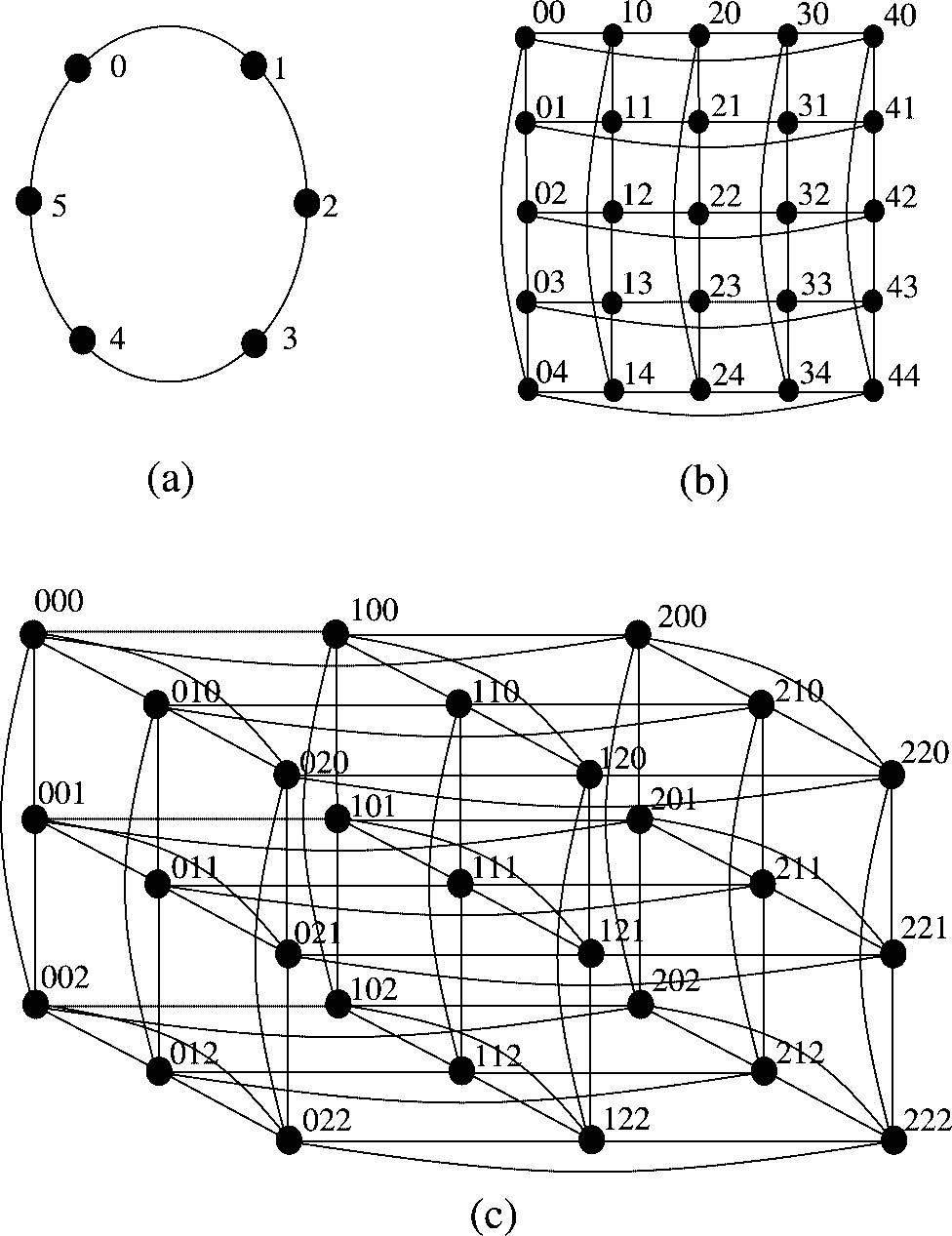}}
\caption{\label{}\small $(a)\  Q_1^6,  (b)\  Q_2^5,\   and\  (c)\  Q_3^3$.}
\end{figure}

There are various works concerning various topological and structural properties of  this class of graphs in the literature, and some of  the recent papers include [4,6,15,17].\

We now wish to determine the automorphism group of the graph $k$-ary   $n$-cube  $ Q_n ^k$ for the case $k \neq 4$,  explicitly.  In particular, we determine the order of the group $Aut(Q_n ^k)$. Also as a consequence, we re-show that  the graph $k$-ary   $n$-cube  $ Q_n ^k$ is a symmetric graph.
In the following theorem   $\Omega=\{ \pm e_i \  | \  1\leq i \leq n \}, $ where  $e_i = (0, ..., 0, 1, 0, ..., 0)$,  with 1    at the $i$-th position.

\begin{thm}   Let $\Gamma= Q_n ^k$ be  the  $k$-ary   $n$-cube. If $k=2$, then $Aut(\Gamma) \cong \mathbb{Z}_2^n \rtimes Sym([n])$, and hence
$|Aut(\Gamma)|=(2^n)(n!)$. If $4 \neq k>2$, then $Aut(\Gamma) \cong \mathbb{Z}_k^n \rtimes A$ and $|Aut(\Gamma)|=(k^n)(2^n)(n!)$, where $A$ is the subgroup of automorphisms  of the group $\mathbb{Z}_k^n$ which fix the set $\Omega$. Moreover, $\Gamma$ is is a symmetric graph.

\end{thm}

\begin{proof}    We know that  $\Gamma= Q_n ^k = Cay(\mathbb{Z}_{k}^n; \Omega )$, where $\mathbb{Z}_{k}$ is
the cyclic group of order $k$, and $\Omega=\{ \pm e_i \  | \  1\leq i \leq n \}, $ where  $e_i = (0, ..., 0, 1, 0, ..., 0)$,  with 1    at the $i$-th position. It is clear   that if $k \neq 4$, then  the graph $\Gamma$  is a $us$-Cayley  graph. Therefore, from Theorem 3.2. we have, $Aut(\Gamma)=L(\mathbb{Z}_{k}^n) \rtimes \cong \mathbb{Z}_{k}^n \rtimes A$,
where $A=\{f| f\in Aut(\mathbb{Z}_{k}^n), f(\Omega) =\Omega \}$. \\
In the first step, let $k=2$. Then, $\Omega=\{  e_i \  | \  1\leq i \leq n \}, $ because in the group $\mathbb{Z}_{2}$ we have $1=-1$, and hence $\Omega$ is an inverse-closed set. If $f\in A$,  then $f|_{\Omega}$, the restriction of $f$ to $\Omega$, is a permutation of the set $\Omega$.  Note  that each element $v$ of the group $\mathbb{Z}_{2}^n$ can be represented uniquely in the form $x=\sum_{i=1}^{n}a_ie_i$, where $a_i \in \mathbb{Z}_2$. Now, it is easy to see that the mapping $\phi : A \rightarrow Sym(\Omega)$, defined by the rule $\phi(f)=f|_{\Omega}$  for every $f\in A$,  is an injection. Hence, we have $|A|\leq n!$. On the other hand, if $g\in Sym(\Omega)$, then we can extend  $g$ linearly to the set $\mathbb{Z}_{2}^n$, and we uniquely obtain an 
automorphism $e_g$ of this group, defined by the rule $e_g(\sum_{i=1}^{i}a_ie_i)=
\sum_{i=1}^{n}a_ig(e_i)$,   such that $e_g|_{\Omega}=g$. This  follows that $|A| \geq n!$. Now  the group $A$ is explicitly determined, that is,  $A=\{e_g | g\in Sym(\Omega)  \}$.  Moreover,   we have $A \cong Sym(\Omega)$. Also, we have
$Aut(\Gamma) \cong \mathbb{Z}_{2}^n \rtimes A \cong \mathbb{Z}_{2}^n \rtimes Sym([n])$. \

We now assume that $k>2$. In this case we have  $\Omega=\{ \pm e_i \  | \  1\leq i \leq n \}$, hence $|\Omega|=2n$.  If $f$ is an automorphism of the group  $\mathbb{Z}_{k}^n$, then $f$ is uniquely determined, whenever $f(e_i)$ is defined for each $i, 1\leq i \leq n$. Note that if $f\in A$, then $g=f|_{\Omega}$ is a permutation of the set $\Omega$ such that when $g(e_i)$ is defined, then $g(-e_i)$ is defined, since $g(-e_i)=-g(e_i)$. Let $B=\{g| g\in Sym(\Omega),  g(-e_i)=-g(e_i), \  $for each$\  i, 1\leq i \leq n \}$. Now, it is easy to see that \
$$|B|=(2n)(2n-2)(2n-4)...(2n-2(n-1))=2^n(n!)$$
Indeed,   for defining $f(e_1)$ we have $2n$ 
choices, then for defining $f(e_2)$ we have $2n-2$ choices,  then for defining $f(e_3)$ we have $2n-4$ choices, and..., and for defining $f(e_{2n})$ we have $2n-2(n-1)=2$ choices.
Note  that each element $v$ of the group $\mathbb{Z}_{k}^n$ can be represented uniquely in the form $x=\sum_{i=1}^{n}a_ie_i$, where $a_i \in \mathbb{Z}_k$. Hence,
if $g \in B$, then $g$ can be uniquely extended to an automorphism $e_g$ of the group $\mathbb{Z}_{k}^n$, by defining the rule $e_g(\sum_{i=1}^{i}a_ie_i)=
\sum_{i=1}^{n}a_ig(e_i)$. Now,  it is clear that $A=\{e_g | g\in B  \}$, and $|A|=|B|$. Also we have;
$$|Aut(\Gamma)|=|\mathbb{Z}_{k}^n||A|=(k^n)(2^n)(n!)$$
We now show that  the graph $\Gamma= Q_n ^k = Cay(\mathbb{Z}_{k}^n, \Omega)$
 is a symmetric graph. Let $G=Aut(\Gamma)$. Since $\Gamma$ is a vertex-transitive graph, it is sufficient to show that $G_0$ acts transitively on the set $N(0)$, where $G_0$ is the stabilizer subgroup of the vertex $v=0$ in the group $G$. Note that $N(0)=\Omega$ and $G_0=A=\{e_g | g\in B  \}$. Let $v,w \in \Omega$. We define $f(v)=w$,  $f(w)=v$, and $f(x)=x$, for every $x \in \Omega -\{ v,w \}$. Hence $f \in B$, and  $e_f$ is an automorphism of the graph $\Gamma$ such that $e_f \in G_0$ with $e_f(v)=w$.
\end{proof}
Theorem 3.7. gives us  some important results.  For 
example,  a well known result due to Watkins [16] states that if a connected graph $\Gamma$ is symmetric, then its vertex and edge connectivity is optimal, namely,  its valency. Hence, we re-obtain the following known fact.
 \begin{cor}  The connectivity of the the   $k$-ary   $n$-cube $Q_n^k$ is $n$ when $k=2$, and is $2n$ when $k>2$.
\end{cor}
 \section{Conclusion}
                                                                                                                                      In this paper, we introduced $us$-Caley graphs and investigated the automorphism groups of these
                                                                                                                                      graphs in Theorem 3.2. Then by Theorem 3.2. we determined the automorphism groups of some
                                                                                                                                      classes of graphs which are important in applied graph theory. In particular, we determined the
                                                                                                                                      automorphism groups of M\"{o}bius ladder $M_n$ (Theorem 3.5) and $k$-ary   $n$-cube  $ Q_n ^k$, $k \neq 4$ (Theorem 3.7).   Moreover, we re-showed that the graph $Q_n^k$ is a symmetric graph, and therefore its connectivity is optimal (Corollary 3.8).

\end{document}